\documentclass[a4paper,11pt]{article}
\usepackage[T1]{fontenc}
\usepackage[utf8]{inputenc}
\usepackage{authblk}

\usepackage[caption=false]{subfig}

\usepackage{hyperref}
\usepackage{diagbox}
\usepackage{multirow}
\usepackage{booktabs}
\usepackage{ulem}
\usepackage{mathptmx}
\usepackage{amsthm,amsmath}
\usepackage{dsfont}
\usepackage{algorithm,algpseudocode}
\usepackage{color}
\usepackage{tikz}
\usepackage{pgfplots}
\usetikzlibrary{calc}
\usepackage{graphicx}
\usepackage{booktabs}
\usepackage{caption}
\usepackage{todonotes}

\newcommand{\R}{\mathds{R}}
\newcommand{\argmin}{\mathop{\mathrm{argmin}}}

\newtheorem{proposition}{Proposition}
\newtheorem{corollary}{Corollary}

\newtheorem{definition}{Definition}

\newtheorem{theorem}{Theorem}

\newcommand{\x}{{\mathbf{x}}}            

\title{A method for global minimization of nonconvex quadratic functions}

\author[1]{A.M. Bagirov}
\author[2]{V. Laha}
\author[3]{J.E. Martínez-Legaz}

\affil[1]{Centre for Smart Analytics, Institute of Innovation, Science and Sustainability, Federation University Australia, Ballarat, Australia}

\affil[2]{Department of Mathematics, Institute of Science, Banaaras Hindu University, Varanasi, 221005, Uttar Pradesh, India}

\affil[3]{Departament d’Economia i d’Historia Economica, Universitat Autonoma de Barcelona, Bellaterra, Spain}

\date{}

\begin{document}

\maketitle

\begin{abstract}
The problem of global minimization of nonconvex quadratic functions subject to box constraints is studied.  Applying Gershgorin theorem we reduce the main matrix $A$ to its diagonal form which is used to obtain difference of convex representation of the quadratic function. Then $\varepsilon$-subdifferentials of DC components are studied and used to design a method for minimizing the quadratic function globally.      
\end{abstract}

\noindent \textbf{Keywords:} Quadratic programming; Global optimization; Nonconvex functions; $\varepsilon$- Subdifferential.


\section{Introduction} \label{sec:Introduction} 
Consider the following \textit{nonconvex quadratic programming (QP) problem}:
\begin{align} \label{globprobQP1}
	\begin{cases}
		\text{minimize}\quad  & f(x)=\frac{1}{2}\langle Ax,x \rangle + \langle b,x \rangle + c\\
		\text{subject to}     & x \in [u,v]. \\
	\end{cases}
\end{align}
Here the matrix $A \in \R^{n\times n}$, $b, u, v \in \R^n,~u_i \leq v_i,~i=1,\ldots,n,~c \in \R$. 

Solving nonconvex quadratic programming problems to global optimality has been the subject of extensive research since the early 1990s. Local methods for constrained nonconvex quadratic programming were developed, for example, in \cite{Cuong2024,Pham2008,Ye1992}. 

In \cite{AlKhayyal1995}, the authors propose an algorithm for computing approximate global solutions to quadratically constrained quadratic programming problems. The approach integrates outer approximation with a branch-and-bound strategy based on linear programming relaxations. The paper \cite{Sherali1995}, introduces a Reformulation-Linearization/Convexification technique to address nonconvex quadratic programming problems subject to linear constraints. A rectangle branch-and-reduce approach for solving nonconvex QP problems is developed in \cite{Gao2005}.

Two approaches were proposed in \cite{Lethi2000} to solve the problem of minimizing a quadratic form subject to the intersection of finitely many ellipsoids where the first approach is based on the DCA algorithm and the second approach is a branch-and-bound scheme using Lagrangian duality for bounding and ellipsoidal bisection in branching. The paper \cite{Audet2000} presents a branch and cut algorithm that finds in a finite number of steps, a globally $\varepsilon$-optimal solution of the nonconvex quadratically constrained QP problem. A decomposition branch and bound method is proposed in \cite{Lethipham1997} to solve linearly constrained indefinite QP problems. The paper \cite{Bomze2004} analyzes DC decompositions for indefinite quadratic functions. 

Branch and bound schemes using a relaxation and different DC decompositions of the objective function, to solve nonconvex QP over a compact polyhedral feasible region is proposed in \cite{Cambini2005}. In \cite{Fampa2017}, a preprocessing stage is proposed to decompose the indefinite-quadratic function into a DC function, leading to a more efficient spatial branch-and-bound. A method to find global minimizers of a nonconvex quadratic function is proposed in \cite{Best1997}. The paper \cite{Fujie1997} applies the semidefinite programming relaxation to a nonconvex QP having a linear objective function and quadratic inequality constraints. The paper \cite{achour2024} proposes an Initialization-free DCA, which automatically calculates the optimal initial point for DCA, enabling efficient solutions to nonconvex quadratic problems with convex constraints.

This paper \cite{luo2023} proposes three algorithms for separable nonconvex QP with a quadratic and a box constraints. The first algorithm improves Lagrangian breakpoint search via a secant approach and the second algorithm uses quadratic convex approximations to converge to KKT points. The third algorithm integrates these methods with convex relaxation and a branch-and-bound framework to obtain global minimizer within a pre-specified $\varepsilon$-tolerance.

The paper \cite{zhao2023global} proposes a Convex Proximal Point Algorithm for  nonconvex QP problems with convex quadratic constraints and develops a special procedure to generate initial points. In \cite{taati2020local}, the authors propose an algorithm to solve nonconvex QP problems with a single quadratic constraint. SDP-based convex relaxation methods for QP problems are discussed in \cite{jiang2019semidefinite}. A branch-and-bound algorithm for solving quadratically constrained QP problems is proposed in \cite{lu2019sensitive}. The paper \cite{wan2019alternating} proposes an iterative algorithm that combines the Alternating Projection Algorithm with the rank-one approximation technique to solve nonconvex QP problems.

Paper \cite{bonami2019solving} introduces novel graph-structure-based cutting planes to tighten the linear relaxation of nonconvex quadratic programming (QP) problems with linear constraints. Leveraging the Motzkin-Straus theorem, a classical result that establishes a connection between graph theory and optimization, the authors employ linearization variables within a spatial branch-and-bound framework to represent bilinear product terms. Similarly,  \cite{bonami2018globally} develops a global optimization approach based on linear programming, incorporating cutting planes derived from the Boolean Quadric Polytope to solve nonconvex QP problems with box constraints. In \cite{lu2018dc}, two branch-and-bound algorithms founded on DC decompositions are proposed to obtain global solutions for box-constrained QP problems. Furthermore, \cite{ge2018accelerating} presents a novel linear relaxation technique based on variable transformation, which reformulates the original nonconvex QP problem into a sequence of computationally efficient linear relaxation problems, thereby generating valid lower bounds on the global optimum. This approach is further enhanced through the integration of region-elimination techniques and a branch-and-bound framework that progressively reduces the search space, improving computational efficiency and facilitating convergence to the global minimizer.

In this paper, we propose different approach to develop a method for globally minimizing indefinite quadratic functions. We apply the Gershgorin theorem to express the quadratic function as a difference of two convex quadratic functions. Then using $\varepsilon$-subdifferentials of DC components we design an hybrid method to minimize the quadratic function. In this hybrid method a local search is used to find stationary points and a global search is applied to escape from stationary points. We prove that after finite number of applications of a local search the method finds global minimizer of the quadratic function.

The paper is organized as follows. Section \ref{prelims} contains definitions and some results necessary for the rest of the paper. DC representation of quadratic functions is presented in Section \ref{diagonal}. $\varepsilon$-subdifferentials of convex quadratic functions are discussed in Section \ref{esubdif}. In Section \ref{method} the proposed method is described and its convergence is studied. Some concluding remarks are given in Section \ref{conclusions}.

\section{Preliminaries} \label{prelims}
In this section, we introduce the definitions and results that will be employed throughout the paper. 

Let $\varphi: \R^n \rightarrow \R$ be a convex function. Its subdifferential $\partial \varphi(x)$ at a point $x$ is:
\begin{align*}
 \partial \varphi(x)=\Big\{\xi \in \R^n:~\varphi(y) \geq \varphi(x)+\langle \xi, y-x\rangle, ~\forall y\in\R^n \Big\}
\end{align*} 
and for $\varepsilon \geq 0$ its $\varepsilon$-subdifferential at $x\in \R^n$ is
\begin{align*}
  \partial_\varepsilon \varphi(x)=\Big\{\xi \in \R^n:~\varphi(y) \geq \varphi(x)+\langle \xi, y-x\rangle-\varepsilon, ~\forall y\in\R^n \Big\}.
\end{align*}
A vector $\xi \in \partial \varphi(x)$ is called a subgradient, and a vector $\xi \in \partial_\varepsilon \varphi(x)$ is called a $\varepsilon$-subgradient of $\varphi$ in $x$. 

The directional derivative of the function $\varphi$ at $x\in \R^n$ with respect to a direction $d \in \R^n$ is defined as
$$
 \varphi'(x,d) = \lim_{t \downarrow 0} \frac{\varphi(x+td)-\varphi(x)}{t}
$$ 
if this limit exists. The directional derivative exists for convex functions $\varphi: \R^n \rightarrow \R$ \cite{bagibook2014} and 
\begin{align} \label{epsdir1}
 \varphi'(x,d) &= \inf_{t > 0}~~t^{-1} \Big[\varphi(x+td)-\varphi(x) \Big].
\end{align}
Furthermore,
\begin{equation} \label{epsdir2}
 \varphi'(x,d)=\max_{\xi \in \partial \varphi(x)}~~\langle \xi,d\rangle.
\end{equation}
The \textit{$\varepsilon$-directional derivative} of a convex function $\varphi$ at $x \in \R^n$ in the direction $d \in \R^n,~d \neq 0_n$, is given by 
\begin{equation} \label{epsdir3}
 \varphi'_\varepsilon(x,d)=\max_{\xi \in \partial_\varepsilon \varphi(x)}~~\langle \xi,d\rangle.
\end{equation}
\noindent It is known that \cite{bagibook2014}
\begin{align} \label{epsdir4}
 \varphi'_\varepsilon(x,d) &=\inf_{t > 0}~t^{-1}\Big[\varphi(x+td)-\varphi(x)+\varepsilon \Big].
\end{align}

Next we recall local and global optimality conditions in unconstrained DC optimization. Consider the DC function $\varphi(x) = \varphi_1(x) - \varphi_2(x)$ where functions $\varphi_1, \varphi_2: \R^n \rightarrow \R$ are convex. The point $x^*$ to be local minimizer of the function $\varphi$ it is necessary that $\partial \varphi_2(x^*) \subseteq \partial \varphi_1(x^*)$. Points satisfying this condition are called \textit{inf-stationary ($d$-stationary, strong critical) points}. In general, this condition is not easy to verify and therefore it is relaxed to the condition: $\partial \varphi_2(x^*) \cap \partial \varphi_1(x^*) \neq \emptyset$. Points satisfying this condition are called \textit{critical points}.

The following theorem on global optimality conditions was established in \cite{Hiriart1988}.
\begin{theorem} \label{hiriart}
 Let $\varphi_1, \varphi_2:\R^n\rightarrow\R$ be convex functions. For a point $x^* \in \R^n$ to be a global minimizer of a DC function $\varphi = \varphi_1 - \varphi_2$, it is necessary and sufficient that
 \begin{equation} \label{hiriart1}
  \partial_\varepsilon \varphi_2(x^*) \subseteq \partial_\varepsilon \varphi_1(x^*) \quad\text{for all}~\varepsilon \geq 0.
 \end{equation}
\end{theorem}

\paragraph{Gershgorin theorem (the symmetric case) \cite{G1931}.} 
Consider the symmetric matrix $A=(a_{ij},~i,j=1,\ldots ,n)$. For every
eigenvalue $\lambda \in \mathds{R},$ there exists $i\in \{1,\ldots ,n\}$
such that%
\begin{equation*}
|a_{ii}-\lambda |\leq \sum_{j}|a_{ij}|\equiv b_{i}.
\end{equation*}

\section{DC representation of quadratic functions} \label{diagonal}
In this section we express the following quadratic function as a difference of two convex quadratic functions:
$$
\varphi(x) = \frac{1}{2}\langle Ax,x \rangle
$$
where $A \in \R^{n \times n}$. 
\begin{proposition}
\label{DCquadratic}Let 
\begin{equation*}
d_{i}=\max \{b_{i}-a_{ii},0\},~i=1,\ldots ,n
\end{equation*}%
where $b_{i}$ is defined in \eqref{DCquadratic}. Then  the matrices%
\begin{equation*}
A_{1}=%
\begin{bmatrix}
a_{11}+d_{1} & a_{12} & \ldots  & a_{1n} \\ 
a_{21} & a_{22}+d_{2} & \ldots  & a_{2n} \\ 
\ldots  & \ldots  & \ldots  & \ldots  \\ 
a_{n1} & a_{2n} & \ldots  & a_{nn}+d_{n} \\ 
&  &  & 
\end{bmatrix}%
\end{equation*}%
and%
\begin{equation*}
A_{2}=%
\begin{bmatrix}
d_{1} & 0 & \ldots  & 0 \\ 
0 & d_{2} & \ldots  & 0 \\ 
\ldots  & \ldots  & \ldots  & \ldots  \\ 
0 & 0 & \ldots  & d_{n} \\ 
&  &  & 
\end{bmatrix}%
\end{equation*}

\noindent are positive semidefinite;\ hence, the function $\varphi $ can be represented
as DC as follows: $\varphi (x)=\varphi _{1}(x)-\varphi _{2}(x)$ where 
\begin{equation*}
\varphi _{1}(x)=\langle A_{1}x,x\rangle ,~\varphi _{2}(x)=\langle
A_{2}x,x\rangle .
\end{equation*}
\end{proposition}
\begin{proof}
The matrix $A$ can be represented as $A=A_{1}-A_{2}.$ Since $d_{i}\geq
0,~i=1,\ldots ,n$ the matrix $A_{2}$ is positive definite. For matrix $A_{1}$
we notice that there exists $i$ such that 
\begin{equation*}
|a_{ii}+d_{i}-\lambda |\leq b_{i}
\end{equation*}%
or 
\begin{equation*}
a_{ii}+d_{i}-\lambda \leq b_{i}
\end{equation*}%
\begin{equation*}
\lambda \geq a_{ii}+d_{i}-\lambda \geq a_{ii}+b_{i}-a_{ii}-b_{i}=0.
\end{equation*}%
This completes the proof.
\end{proof}

\section{$\varepsilon$-subdifferentials of convex quadratic functions.} \label{esubdif}
This section is devoted to the study of $\varepsilon$-subdifferentials of convex quadratic functions. 

\paragraph{$\varepsilon$-subdifferentials.} First, we describe the $\varepsilon$-subdifferentials of the following quadratic function
\begin{align} \label{quadratic}
 f(x) = \frac{1}{2}\langle Ax,x \rangle + \langle b,x \rangle + c
\end{align}
where the matrix $A \in \R^{n \times n}$ is symmetric positive definite and invertible, $b \in \R^n,~c \in \R$. 

\begin{proposition} \label{epssub01}
	The $\varepsilon$-subdifferential of the function $f$, defined in \eqref{quadratic}, at $x \in \R^n$ is:
	\begin{align} \label{epsilonsub}
		\partial_\varepsilon \varphi(x)=\Big\{\xi \in \R^n:~\xi = Ax+b+y,~\langle A^{-1}y,y \rangle \leq 2\varepsilon \Big\}.
	\end{align}
	Here $A^{-1}$ is an inverse of the matrix $A$.
\end{proposition}
\begin{proof}
	The gradient of the function $\varphi$ is $\nabla \varphi(x) = Ax+b$. We represent any $\varepsilon$-subgradient $\xi \in \partial_\varepsilon \varphi(x)$ as $\xi = \nabla \varphi(x) + y,~y \in \R^n$. It follows from the definition of $\varepsilon$-subgradients that
	$$
	\frac{1}{2}\Big[\langle Az,z \rangle-\langle Ax,x \rangle\Big]+\langle b, z-x \rangle \geq \langle Ax+b+y, z-x \rangle-\varepsilon, ~\forall z \in \R^n.
	$$
Simplifying this, we get
	$$
	\frac{1}{2} \Big[\langle Az,z \rangle+\langle Ax,x \rangle \Big]-\langle Ax,z \rangle \geq \langle y, z- x \rangle - \varepsilon.
	$$
	This implies that
	\begin{equation} \label{inequality}
	  \frac{1}{2} \langle A(z-x), z-x \rangle-\langle y, z-x \rangle+ \varepsilon \geq 0.
	\end{equation}
	Next we find a gradient of the left hand side with respect to $z$ and set it to zero:
	$$
	A(z^*-x) - y = 0.
	$$
	Since the matrix $A$ is invertible it follows from here that $z^* = x+ A^{-1}y$. It follows from \eqref{inequality} that
	$$
	-\frac{1}{2} \langle y, A^{-1}y \rangle + \varepsilon \geq 0.
	$$	
	From here we get $\langle A^{-1}y, y \rangle \leq 2\varepsilon.$ This completes the proof. 
\end{proof}

Proposition \ref{epssub01} demonstrates how to estimate $\varepsilon$-subdifferentials of convex quadratic functions.  
\begin{corollary} \label{eigenvect}
Consider the function
$$
\psi(x) = \frac{1}{2}\sum_{i=1}^n d_ix_i^2.
$$
The $\varepsilon$-subdifferential of the function $\psi$ at $x \in \R^n$ is:
\begin{align} \label{epsilonsub}
\partial_\varepsilon \psi(x)=\Big\{\xi \in \R^n:~\xi = (d_1x_1,\ldots, d_nx_n)+y,~\sum_{i=1, d_i >0}^n d_i^{-1} y_i^2 \leq 2\varepsilon,~y_i=0 ~\mbox{if}~d_i=0 \Big\}.
\end{align}
\end{corollary}

\section{The proposed method} \label{method}
The proposed method for solving Problem \eqref{globprobQP1} is an hybrid method based on the local search method and a special procedure for escaping local solutions. We take any starting point point $x_0 \in \R^n$ and apply the local search method to find stationary point of the quadratic function. For example, the conjugate gradient method can be used to find such points. Then we apply the escaping procedure which is designed using $\varepsilon$-subdifferentials of the DC components. 

The escaping procedure finds the so-called global descent directions from stationary points which are not global minimizers. The notion of the global descent direction was introduced in \cite{BagJokMakTah2026}:

\begin{definition} \label{globdescent}
Let $x \in \R^n$ be a local minimizer of the problem \eqref{globprobQP1}. A direction $d \in \R^n, d \neq 0$ is called a global descent direction of the function $f$ at the point $x$ if there exist $\bar{\alpha}>0$ and $\delta \in
(0, \bar{\alpha})$ such that $f(x + \alpha d) < f(x)$ for all $\alpha \in (\bar{\alpha} - \delta, \bar{\alpha} + \delta)$.
\end{definition}
Note that global descent directions are defined at local minimizers. The following theorem was proved in \cite{BagJokMakTah2026}.
\begin{theorem} \label{globopt1}
    Let $f:\R^n \rightarrow \R$ be a DC function: $\varphi(x)=\varphi_1(x)-\varphi_2(x)$ where functions $\varphi_1, \varphi_2: \R^n \rightarrow \R$ are convex. Let $x \in \R^n$ be a local minimizer of $\varphi$, $\partial_\varepsilon \varphi_i(x)$ be $\varepsilon$-subdifferentials of the functions $\varphi_1$ and $\varphi_2$, respectively. Assume that $\partial_\varepsilon \varphi_2(x) \not\subset \partial_\varepsilon \varphi_1(x)$ for some $\varepsilon > 0$ and let $\xi_2 \in \partial_\varepsilon \varphi_2(x)$ be such that $\xi_2 \notin \partial_\varepsilon \varphi_1(x)$. Define $\xi_1 \in \partial_\varepsilon \varphi_1(x)$ such that
    $$
    \xi_1 = \argmin_{\xi \in \partial_\varepsilon \varphi_1(x)} \|\xi-\xi_2\|. 
    $$
    Then the direction $d=-(\xi_1-\xi_2)$ is a global descent direction at $x$.
\end{theorem}
Consider the problem \eqref{globprobQP1}. According to Proposition \ref{DCquadratic} the objective function $f$ can be represented as a difference of two convex quadratic function: $f(x) = f_1(x) - f_2(x)$. Using the exact penalty function approach we can include box-constraints into the first DC component by writing 
$$
f_1(x) + \gamma \max\Big\{0, u_i - x_i, x_i - v_i:~i = 1,\ldots,n \Big\} 
$$
which is still convex. Here $\gamma > 0$ is the penalty parameter. We will use the same notation for the function $f(x)=f_1(x)-f_2(x)$ and assume that 
$$
f_* = \inf \Big \{f(x):~x \in \R^n \Big\} > -\infty.
$$
Now we are ready to describe the proposed method.
\begin{algorithm} [H]
	\caption{Hybrid method for global minimization of quadratic functions: conceptual version.} \label{concepthybrid}
	\smallskip
	\begin{algorithmic}[1]
	   \State\label{chmsinitialc} \textit{(Initialization)} Select the initial point $x_0 \in \R^n$. Set $\bar{x}_0 = x_0$ and $k=0$.
	   \smallskip
      \State\label{clocalc} \textit{(Local minimizer)} Apply a local method starting from $\bar{x}_k$ to find a local minimizer $x_k$.
	   \smallskip
	   \State\label{chsetcalc} For each $\varepsilon > 0$, compute the $\varepsilon$-subdifferentials $\partial_\varepsilon f_1(x_k)$ and $\partial_\varepsilon f_2(x_k)$. 
      \smallskip
      \State\label{globoptimality} If $\partial_\varepsilon f_2(x_k) \subseteq \partial_\varepsilon f_1(x_k)$, then STOP. The point $x_k$ is a global minimizer. 
      \smallskip
      \State\label{globdescent} Find $\varepsilon > 0$, $\bar{\xi}_{1k} \in \partial_\varepsilon f_1(x_k)$ and $\bar{\xi}_{2k} \in \partial_\varepsilon f_2(x_k)$ such that
      $$
       \|\bar{\xi}_{1k}-\bar{\xi}_{2k}\| = \max_{\xi_2 \in \partial_\varepsilon f_2(x_k)} \min_{\xi_2 \in \partial_\varepsilon f_2(x_k)} \|\xi_1-\xi_2\| >0
      $$
      \smallskip
      \State\label{linesearch} Construct the direction $\bar{d}_k=-(\bar{\xi}_{1k}-\bar{\xi}_{2k})$ and find $\bar{\alpha}_k>0$ such that
      $$
       \bar{\alpha}_k=\argmin_{\alpha>0}\Big\{f(x_k+\alpha d_k)-f(x_k) \Big\}.
      $$
      \smallskip
      \State\label{update} Set $\bar{x}_{k+1}=\x_k+\bar{\alpha}_k d_k$, $k=k+1$ and go to Step \ref{clocalc}.    
	\end{algorithmic}
\end{algorithm}

Next, we present the convergence result for Algorithm \ref{concepthybrid}.
\begin{theorem}
 Let $C$ be a nonempty set containing all stationary points of the problem \eqref{globprobQP1}. Assume that the set of objective function values
 $$
  V = \Big\{v \in \R:~~\exists~x \in C~\mbox{such~that}~v=f(x) \Big\}
 $$
 at these stationary points is finite. Then Algorithm \ref{concepthybrid} in a finite number of iterations finds a global minimizer of the problem \eqref{globprobQP1}.
\end{theorem}
\begin{proof}
    Algorithm 1 finds a new stationary point at each iteration. According to Theorem \ref{globopt1}, the value of the objective function at this stationary point is strictly less than at the stationary point found in the previous iteration. Since the number of such values is finite and the function $f$ is bounded from below on compact sets the algorithm will find a global minimizer after a finite number of iterations.
\end{proof}




\section{Conclusions} \label{conclusions}
In this paper, we develop a new method for global minimization of quadratic functions subject to box-constraints. Using the Gershgorin theorem we represent the indefinite quadratic function as the difference of two convex quadratic functions. Next we formulate $\varepsilon$-subdifferentials of DC components and design a method for finding global minimizers of the quadratic functions. This method is based on the necessary and sufficient global optimality condition reformulated for quadratic functions. We discuss the convergence of the proposed method.

\vspace{1cm}

\noindent \textbf{Acknowledgment.} Juan Enrique Martínez-Legaz has been partially supported by Grant PID2022-136399NB-C22 from MICINN, Spain, and ERDF, ``A way to make Europe", European Union.

\bibliography{QP}
 
\end{document}